\documentclass[preprint,12pt]{elsarticle}

\usepackage{amssymb}
\usepackage{graphicx}
\usepackage{epsfig}
\usepackage{enumitem}

\biboptions{numbers,sort&compress}

\newcommand{\C}{\Gamma}

\newcommand{\M}{{\cal M}}
\newcommand{\e}{\varepsilon}
\newcommand{\f}{\varphi}
\newcommand{\eR}{\mathbb{R}}
\newcommand{\eN}{\mathbb{N}}
\newcommand{\Ze}{\mathbb{Z}}

\begin{document}

\begin{frontmatter}

\title{Wavy spirals and their fractal connection with chirps}

\author[fer]{L.\ Korkut}
\ead{luka.korkut@fer.hr}

\author[fer]{D.\ Vlah\corref{corresp}}
\ead{domagoj.vlah@fer.hr}

\author[fer]{D.\ \v Zubrini\'c}
\ead{darko.zubrinic@fer.hr}

\author[fer]{V.\ \v Zupanovi\'c}
\ead{vesna.zupanovic@fer.hr}

\cortext[corresp]{Corresponding author. Tel. +385 01 612 99 03, Fax +385 01 612 99 46}

\address[fer]{University of Zagreb, Faculty of Electrical Engineering and Computing, Unska 3, 10000 Zagreb, Croatia}

\journal{arXiv.org}

\begin{abstract}
We study the fractal oscillatority of a class of real $C^1$ functions $x=x(t)$ near $t=\infty$. It is measured  by \emph{oscillatory} and \emph{phase dimensions}, defined as box dimensions of the graph of $X(\tau)=x(\frac{1}{\tau})$ near $\tau=0$ and the trajectory $(x,\dot{x})$ in $\mathbb{R}^2$, respectively, assuming that $(x,\dot{x})$ is a  spiral converging to the origin. The relationship between these two dimensions has been established for a class of oscillatory functions using formulas for box dimensions of graphs of chirps and nonrectifiable wavy spirals, introduced in this paper. \emph{Wavy spirals} are a specific type of spirals, given in polar coordinates by $r=f(\f)$, converging to the origin in \emph{non-monotone} way as a function of $\f$. They emerged in our study of phase portraits associated to solutions of Bessel equations. Also, the rectifiable chirps and spirals have been studied.
\end{abstract}


\end{frontmatter}

\newtheorem{theorem}{Theorem}
\newtheorem{proposition}{Proposition}
\newtheorem{lemma}{Lemma}
\newdefinition{definition}{Definition}
\newdefinition{remark}{Remark}
\newproof{proof}{Proof}

\section{Introduction}
Fractal analysis of differential equations has since the last decades emerged as an important tool in better understanding the behavior of their oscillatory solutions. The main focus of fractal analysis in dynamics is in fractal dimension theory. Its goal is in determining complexity of invariant sets and measures using fractal dimensions. The fractal dimension has been successfully used in studying, for instance, logistic map, Smale horseshoe, Lorenz attractor, H\'{e}non Attractor, Julia and Mandelbrot sets, spiral trajectories, infinite-dimensional dynamical systems and even in probability theory, see \cite{fdd}.

In this work we are focused on studying the connection between the fractal dimension of graphs of oscillatory solutions, and the fractal dimension of the associated phase portraits. In particular, we use the box dimension, which we exploit instead of Hausdorff dimension. Due to the countable stability of Hausdorff dimension, its value is trivial on all smooth nonrectifiable curves, while the box dimension is nontrivial, that is, larger than $1$. From the view of fractal analysis of trajectories and graphs of solutions of differential equations, the most interesting are solutions having phase plots and graphs of infinite length. The Hausdorff dimension is not suitable to classify these solutions, while the box dimension comes in handy.

Our work was initially inspired by Tricot \cite{tricot}, where box dimensions of graphs of a simple spiral ($r=\f^{-\alpha}$, $\alpha\in(0,1)$, in polar coordinates) and $(\alpha,\beta)$-chirp ($f(t)=t^{\alpha}\cos{t^{-\beta}}$, $0<\alpha<\beta$) have been computed near the origin. Since then, these results have been generalized to some more general spiral trajectories of dynamical systems, and to chirp-like functions. Fractal properties of spiral trajectories of dynamical systems in the phase plane have been studied by \v Zubrini\'c and \v Zupanovi\'c, see e.g.\ \cite{zuzu, zuzulien, cras}. An interesting behavior of the box dimension of spiral trajectories has been found and related to the bifurcation of a system, in particular to the Hopf bifurcation. On the other hand, the chirp-like behavior of solutions of different types of second-order linear differential equations is also of interest. It has been studied by Kwong, Pa\v si\'c, Tanaka and Wong: Euler type equations are considered in \cite{pasiceuler, pasic, wong1}, Hartman-Wintner type equations in \cite{kpw}, half linear equations in \cite{pasicwo}, and for the Bessel equation see \cite{mersat}. More specifically, this work has been motivated by Pa\v si\'c, \v Zubrini\'c and \v Zupanovi\'c \cite{chirp}, containing the first results connecting fractal properties of chirps and spirals, with applications to Li\' enard and Bessel equations.

All this encouraged us to study and analyze the connection between chirp-like functions and the corresponding spiral trajectories in the phase plane and vice versa. There are two possible ways of looking at the solutions: using the graph of a solution, or using the phase plot of the solution, and the latter was first theoretically developed by Poincar\'{e}. Our main results are obtained in Theorems \ref{s-c-p} and \ref{c-s}. For some applications to differential equations and dynamical systems, see \cite{bessel,luvedo}.

A specific type of spirals associated to oscillatory solutions of Bessel equations, emerged in our study of phase portraits, converging to the origin in \emph{non-monotone} way as a function of $\f$. We call them \emph{wavy spirals}, see Definition \ref{def_wavy_spiral}. They also appear in the study of curves given by the parametrization of oscillatory integrals from Arnol$'$d, Guse{\u\i}n-Zade and Varchenko, \cite[Part II]{arnold-vol2}. These curves can exhibit even more complex behavior, having self-intersections. The oscillatory integrals from \cite{arnold-vol2} are naturally related to generalized Fresnel integrals, and the fractal properties of associated spirals studied in \cite{q_clothoid}.

Techniques of fractal analysis have also been successfully applied in the study of bifurcations: see e.g.\ Horvat Dmit\-ro\-vi\'c \cite{laho}, Li and Wu \cite{li}, Marde\v si\' c, Resman and \v Zupanovi\'c \cite{MRZ}, Resman \cite{majaformal}, as well as in the case of Hopf bifurcation at infinity considered in Radunovi\'c, \v Zubrini\'c and \v Zupanovi\'c \cite{RaZuZu}, and for infinite-dimensional dynamical systems related to Schr\"{o}dinger equation, see Mili\v si\'c, \v Zubrini\'c and \v Zupanovi\'c \cite{mzz-Schrodinger}.

\section{Summary of results and definitions}

Our results can be summarized in the following way. An $(\alpha,1)$-chirp-like function near the origin, described in Theorem \ref{s-c-p}, is connected with a spiral ``similar'' to $r=\f^{-\alpha}$, defined in polar coordinates. If $\alpha\in(0,1)$ then the box dimension of this spiral is equal to $\frac{2}{\alpha+1}$, see Theorems \ref{novizuzu} and \ref{s-c-p}(i), and for $\alpha>1$ this spiral is rectifiable, see Theorem \ref{s-c-p}(ii). Furthermore, we consider the opposite direction and generate a chirp from a given  planar spiral. The obtained chirp is $(\alpha,1)$-chirp-like function, $\alpha\in(0,1)$, see Definition \ref{def_chirp_like_function}, and the box dimension of its graph is equal to $\frac{3-\alpha}{2}$, see Theorems~\ref{BDchirp} and \ref{c-s}. If the planar spiral is rectifiable and $\alpha>1$, then the corresponding chirp-like function has rectifiable graph as well, see Theorem \ref{eqcs1}.

\smallskip

Let us introduce some definitions and notation.
Given a bounded subset $A$ of $\eR^N$, we define the  $\e$-neighborhood of $A$ by
$
A_\e:=\{y\in\eR^N\:d(y,A)<\e\}
$,
where $d(y,A)$ denotes the Euclidean distance from $y$ to $A$.
By {\it lower $s$-dimensional  Minkowski content} of $A$, $s\ge0$ we mean
$$
\M_*^s(A):=\liminf_{\e\to0}\frac{|A_\e|}{\e^{N-s}} ,
$$
and analogously for the {\it upper $s$-dimensional Minkowski content} $\M^{*s}(A)$. If both these quantities coincide,
the common value is called the {\it s-dimensional} {\it Minkowski content of A},
and is denoted by $\M^s(A)$.
Now we can introduce the {\it lower and upper box dimensions} of $A$ by
$$
\underline\dim_BA:=\inf\{s\ge0\:\M_*^s(A)=0\} ,
$$
 and analogously
$\overline{\dim}_BA:=\inf\{s\ge0\:\M^{*s}(A)=0\}$.
If these two values coincide, we call it simply the box dimension of $A$, and denote it by $\dim_BA$.

If $0<\M_*^d(A)\le\M^{*d}(A)<\infty$ for some $d$, then we say
 that $A$ is {\it Minkowski nondegenerate}. In this case obviously $d=\dim_BA$.
In the case when lower or upper $d$-dimensional Minkowski contents of $A$ are $0$ or $\infty$,
where $d=\dim_BA$, or $\underline{\dim}_BA<\overline{\dim}_BA$,
we say that $A$ is {\it degenerate}.

More details on these definitions can be seen in  Falconer \cite{falc} and Tricot \cite{tricot}. Some generalizations can be seen in \cite{MRZ}.

\begin{definition} Let $x:[t_0,\infty)\to\eR$, $t_0>0$, be a continuous function. We say that $x$ is {\it oscillatory function} near $t=\infty$ if there exists a sequence $t_k\to\infty$ such that $x(t_k)=0$, and the functions $x|_{(t_k,t_{k+1})}$ alternately change sign for $k\in \eN$.
\end{definition}

Analogously, let  $u:(0,t_0]\to\eR$, $t_0>0$, be a continuous function. We say that $u$ is {\it oscillatory function} near the origin if there exists a sequence $s_k$ such
that $s_k\searrow 0$ as $k\rightarrow \infty $, $u(s_k)=0$ and restrictions $u|_{(s_{k+1},s_k)}$ alternately change sign for $k\in \eN$.

\begin{definition} (see Pa\v si\' c \cite{pasic}) Suppose that $v:I\to {\eR},\ I=(0,1],$ is an oscillatory function near the origin, $d\in[1,2)$. We say that $v$ is {\it{d-dimensional fractal oscillatory}} near the origin if $\dim_BG(v)=d$ and $0<\M_*^{d}(G(v))\le\M^{*d}(G(v)) <\infty$, where $G(v)$ denotes the graph of~$v$.
\end{definition}

\begin{definition} Assume that $x:[t_0,\infty)\to\eR$ is oscillatory near $t=\infty$. Let us define $X:(0,1/t_0]\to\eR$ by $X(\tau)=x(1/\tau)$. It is clear that $X(\tau)$ is oscillatory near the origin. We measure the rate of oscillatority of $x(t)$ near $t=\infty$ by the rate of oscillatority of $X(\tau)$ near $\tau=0$. More precisely, the {\it oscillatory dimension} $\dim_{osc}(x)$ (near $t=\infty$) is defined as box dimension
of the graph of $X(\tau)$ near $\tau=0$:
$$
\dim_{osc}(x)=\dim_B G(X),
$$
provided the box dimension exists.
\end{definition}

\begin{definition}Assume now that $x$ is of class $C^1$. We say that $x$ is a {\it phase oscillatory} function if the following condition holds: the set $\C=\{(x(t),\dot x(t)):t\in[t_0,\infty)\}$ in the plane is a spiral converging to the origin.
\end{definition}

\begin{definition}By a {\it spiral} here we mean the graph of a function $r=f(\varphi)$, $\varphi\geq\varphi_1>0$, in polar coordinates, where
\begin{equation}\label{def_spiral}
\left\{\begin{array}{l}
f:[\varphi_1,\infty)\rightarrow(0,\infty) \textrm{ is such that } f(\varphi)\to 0 \textrm{ as } \varphi\to\infty,\\
f \textrm{ is \emph{radially decreasing} (i.e., for any fixed } \varphi\geq\varphi_1\\
\textrm{the function } \mathbb{N}\ni k\mapsto f(\varphi+2k\pi) \textrm{ is decreasing)} .
\end{array}\right.
\end{equation}
\end{definition}

This definition appears in \cite{zuzu}. Depending on the context, by a spiral here we also mean the graph of a function $r=g(\f)$, $\f\leq\f'_1<0$, in polar coordinates,
such that the curve defined as the graph of $r=g(-\f)$, $\f\geq|\f'_1|>0$, given in polar coordinates, satisfies (\ref{def_spiral}). It is easy to see that a spiral defined
by a function $g(\f)$ is a mirror image of the spiral defined by $g(-\f)$, with respect to the $x$-axis. We also say that a graph of a function $r=f(\varphi)$, $\varphi\geq\varphi_1>0$,
defined in polar coordinates, is a {\it spiral near the origin} if there exists $\f_2\geq\f_1$ such that the graph of the function $r=f(\varphi)$, $\varphi\geq\varphi_2$, viewed  in polar coordinates, is a spiral.

\begin{definition}The {\it phase dimension} $\dim_{ph}(x)$ of a function $x:[t_0,\infty)\to\eR$ of class $C^1$ is defined as the box dimension of the corresponding planar curve $\C=\{(x(t),\dot x(t)):t\in[t_0,\infty)\}$.
\end{definition}

Oscillatory and phase dimensions are fractal dimensions, introduced in the study of  chirp-like solutions of second order ODEs, see \cite{chirp}. More about fractal dimensions in dynamics can be found in \cite{fdd}.

\smallskip

For two real functions $f(t)$ and $g(t)$ of a real variable
we write $f(t)\simeq g(t)$ as $t\to0$ (as $t\to\infty$)
if there exist two positive constants $C$ and $D$ such that $C\,f(t)\le g(t)\le D\,f(t)$ for all $t$ sufficiently close to $t=0$ (for all $t$ large enough).
For a function $F:U\to V$, with $U,V\subset\eR^2$, $V=F(U)$, we write $|F(x_1)-F(x_2)|\simeq
|x_1-x_2|$ if $F$ is a bi-Lipschitz mapping, i.e.,\ both $F$ and $F^{-1}$ are Lipschitzian.

\begin{definition}We write $f(t)\sim g(t)$ if $f(t)/g(t)\to1$ as $t\to0$ (as $t\to\infty$). Also, if $k$  is fixed positive integer, for two functions $f$ and $g$ of class $C^k$ we write,
$$
f(t)\sim_k g(t) \ \mbox{as}\ t\to0\ (\mbox{as}\ t\to\infty) ,
$$
if $f^{(j)}(t)\sim g^{(j)}(t)$ as $t\to 0$ (as $t\to\infty$) for all $j=0,1,...,k$.
\end{definition}

For example, $\frac{(t-1)^{4-\alpha}}{t^4}\sim_3t^{-\alpha}$ as $t\to\infty$, for $\alpha\in(0,1)$.

\smallskip

Analogously,   if $k$  is fixed positive integer, for two functions $f$ and $g$ of class $C^k$ we write
$$
f(t)\simeq_kg(t)\ \mbox{as}\ t\to 0\ (\mbox{as}\ t\to\infty),
$$
if $f^{(j)}(t)\simeq g^{(j)}(t)$ as $t\to 0$ (as $t\to\infty$) for all $j=0,1,...,k$.

\smallskip

We write $f(t)=O(g(t))$ as $t\to0$ (as $t\to\infty$) if there exists a positive constant $C$ such that $|f(t)|\leq C|g(t)|$ for all $t$ sufficiently close to $t=0$. (for all $t$ large enough). Similary, we write $f(t)=o(g(t))$ as $t\to\infty$ if for every positive constant $\varepsilon$ it holds $|f(t)|\leq \varepsilon|g(t)|$ for all $t$ sufficiently large.

\begin{definition}\label{def_chirp_like_function} Functions of the form
$$y=P(x)\sin(Q(x))\ \mathrm{or}\ y=P(x)\cos(Q(x)) ,$$
where $P(x)\simeq x^\alpha$, $Q(x)\simeq_1 x^{-\beta}$ as $x\to 0$, are called $(\alpha,\beta)$-{\it chirp-like function} near $x=0$.
\end{definition}

\section{Spirals generated by chirps}

We study spirals generated by chirps in the sense of Theorem \ref{s-c-p}. To prove Theorem \ref{s-c-p} about box dimension of a spiral generated by a chirp we need a new version of \cite[Theorem~5]{zuzu}. Let us first recall \cite[Theorem 5]{zuzu}, cited here in a more condensed form, suitable for our purposes. The following theorem extends a result about box dimension of spiral from Dupain, Mend\`es France and Tricot, see \cite{tric, tricot}.

\begin{theorem}[Theorem 5 from \cite{zuzu}]{\label{ffa}}
Let $f\:[\f_1,\infty)\to(0,\infty)$ be a decreasing function of class $C^2$, such that $f(\varphi)\to 0$ as $\varphi\to\infty$. Let $\alpha\in(0,1)$. Assume that there exist positive constants $\underline{m}$, $\overline{m}$, $M_1$, $M_2$ and $M_3$ such that for all $\f\ge\f_1>0$,
$$
\underline{m}\,\f^{-\alpha}\le f(\f)\le\overline{m}\,\f^{-\alpha} ,
$$
$$
M_1\f^{-\alpha-1}\le|f'(\f)|\le M_2\f^{-\alpha-1},\quad |f''(\f)|\le M_3\f^{-\alpha}.
$$
Let $\C$ be the graph of $r=f(\f)$ in polar coordinates.
Then
$$
\dim_B\C=\frac{2}{1+\alpha} .
$$
\end{theorem}

Now we provide a new version of Theorem \ref{ffa}.

\begin{theorem}[Dimension of a piecewise smooth nonincreasing spiral]{\label{novizuzu}}
Let $f:[\f_1,\infty)\rightarrow(0,\infty)$ be a nonincreasing and radially decreasing function, also a continuous and piecewise continuously differentiable. We assume that the number of smooth pieces of $f$ in $[\f_1,\overline{\f}_1]$ is finite, for any $\overline{\f}_1>\f_1$.
Assume that there exist positive constants $\underline{m}$, $\overline{m}$, $\underline{a}$ and $M$ such that for all $\f\geq\f_1$,
$$
\underline{m}\f^{-\alpha}\leq f(\f)\leq \overline{m}\f^{-\alpha} ,
$$
$$
\underline{a}\f^{-\alpha-1}\leq f(\f)-f(\f+2\pi) ,
$$
and for all $\f$ where $f(\f)$ is differentiable,
$$
|f'(\f)|\leq M\f^{-\alpha-1} .
$$
Let $\Gamma$ be the graph of $r=f(\f)$ in polar coordinates. If $\alpha\in(0,1)$ then
$$
\dim_B\Gamma=\frac{2}{1+\alpha}.
$$
\end{theorem}

For the proof of Theorems \ref{novizuzu} and \ref{s-c-p} below, we need the following Lemma~\ref{novalemazuzu} that is a generalization of \cite[Lemma 1]{zuzu} dealing with smooth spirals.

\begin{lemma}[Excision property for piecewise smooth curves]{\label{novalemazuzu}}
Let $\Gamma$ be the image of a continuous and piecewise continuously differentiable
function $h:[\f_1,\infty)\to\mathbb{R}^2$ $($piecewise in the sense of Theorem \ref{novizuzu}$)$. Assume that $\underline{\dim}_B\Gamma>1$, $\Gamma_1:=h((\overline{\f}_1,\infty))$, for some fixed $\overline{\f}_1>\f_1$, and $h([\f_1,\overline{\f}_1])\bigcap\Gamma_1=\emptyset$. Then
$$
\underline{\dim}_B\Gamma_1=\underline{\dim}_B\Gamma,\quad \overline{\dim}_B\Gamma_1=\overline{\dim}_B\Gamma .
$$
\end{lemma}

\begin{proof}
The proof is analogous to the proof of \cite[Lemma 1]{zuzu}, but with the following difference.
Here, the curve $\Gamma_2:=\Gamma\backslash\Gamma_1=h([\varphi_1,\overline{\varphi}_1])$ is rectifiable due to piecewise rectifiability of $h$ and due to the finite number of pieces in segment $(\varphi_1,\overline{\varphi}_1]$. The function $h$ is piecewise rectifiable due to its piecewise smoothness and continuity. Also, by careful examination of the proof of \cite[Lemma 1]{zuzu}, it follows that we can substitute the injectivity assumption on $h$ with the weaker condition that $h([\f_1,\overline{\f}_1])\bigcap\Gamma_1=\emptyset$. (For more details see \cite[Lemma~1]{zuzu}.)
\qed
\end{proof}

\begin{proof}[Theorem \ref{novizuzu}]
The proof is analogous to the proof of \cite[Theorem~5]{zuzu}, but using the new Lemma \ref{novalemazuzu}.
\qed
\end{proof}

\begin{remark}
Notice the difference between the assumptions of Theorem \ref{ffa} and Theorem \ref{novizuzu}. In Theorem \ref{ffa} the function $f$ is decreasing and of class $C^2$. By careful examination of the proof of \cite[Theorem 5]{zuzu}, one can see that $f$ being decreasing is used only in the sense of nonincreasing, that is, not strictly decreasing, hence in Theorem \ref{ffa} we can assume that $f$ is nonincreasing. Additional smoothness of $f$ and additional conditions regarding constants $M_1$ and $M_3$ in Theorem \ref{ffa} are used only in the calculations of Minkowski contents in \cite[Theorem 5]{zuzu} which we exclude from our Theorem \ref{novizuzu}. Further reduction in smoothness of $f$ from continuously differentiable to a piecewise continuously differentiable function can be found in Lemma \ref{novalemazuzu}.
\end{remark}

Theorem \ref{biliplema} deals with a spiral $\Gamma'$ described by $r=f(\f)$, where $f$ is increasing on some parts, see Definitions \ref{def_wavy_function} and \ref{def_wavy_spiral}. We call this new property of $\Gamma'$ \emph{spiral  waviness}.

\begin{definition}\label{def_wavy_function} Let $r:[t_0,\infty)\rightarrow(0,\infty)$ be a $C^1$ function. Assume that $r'(t_0)\leq 0$. We say that $r=r(t)$ is a \emph{wavy function} if the sequence $(t_n)$ defined inductively by:
\begin{eqnarray}
t_{2k+1} & := & \inf\{t : t>t_{2k}, r'(t)>0\},\quad k\in\mathbb{N}_0 , \nonumber\\
t_{2k+2} & := & \inf\{t : t>t_{2k+1}, r(t)=r(t_{2k+1})\},\quad k\in\mathbb{N}_0 , \nonumber
\end{eqnarray}
is well-defined, and satisfies the \emph{waviness condition}:
\begin{equation}\label{def_SRC}
\left\{
\begin{minipage}{10cm}
\begin{itemize}
\item[(i)] The sequence $(t_n)$ is increasing and $t_n\to\infty$ as $n\to\infty$.
\item[(ii)] There exists $\varepsilon>0$, such that for all $k\in\mathbb{N}_0$ holds $t_{2k+1} - t_{2k} \geq\varepsilon$.
\item[(iii)] For all $k$ sufficiently large it holds
    $\mathop{\mathrm{osc}}\limits_{t\in[t_{2k+1},t_{2k+2}]} r(t) = o\left(t_{2k+1}^{-\alpha-1}\right)$,\newline $\alpha\in(0,1)$,
\end{itemize}
\end{minipage}
\right.
\end{equation}
where $\mathop{\mathrm{osc}}\limits_{t\in I} r(t)=\max\limits_{t\in I} r(t) - \min\limits_{t\in I} r(t)$.
\end{definition}

Notice that $\min\limits_{t\in[t_{2k+1},t_{2k+2}]} r(t)=r(t_{2k+1})$. Condition $(i)$ means that the property of waviness of $r=r(t)$ is global on the whole domain.
Condition~$(ii)$ is connected to an assumption of Lemma~\ref{tehnicka2}.
Condition $(iii)$ is a condition on a decay rate on the sequence of oscillations of $r$ on $I_k=[t_{2k+1},t_{2k+2}]$, for $k$ sufficiently large. Also, notice that condition $r'(t_0)\leq 0$ assures that $t_1$ is well-defined. For an example of function $r(t)$, see Figure \ref{fig:funkcija_r}.

\begin{remark}\label{remark-tehnicki-uvjet}
Conditions (i) and (ii) in the waviness condition (\ref{def_SRC}) are not entirely independent. From (ii) and if $(t_n)$ is increasing follows that $t_n\to\infty$ as $n\to\infty$, but from (i) does not follow (ii). So, condition (ii) plus  $(t_n)$ increasing is stronger than condition (i).
\end{remark}

\setlength{\unitlength}{0.9\textwidth}
\begin{figure}
\begin{center}
\begin{picture}(1,0.64)
\put(0,0){\includegraphics[width=0.9\textwidth]{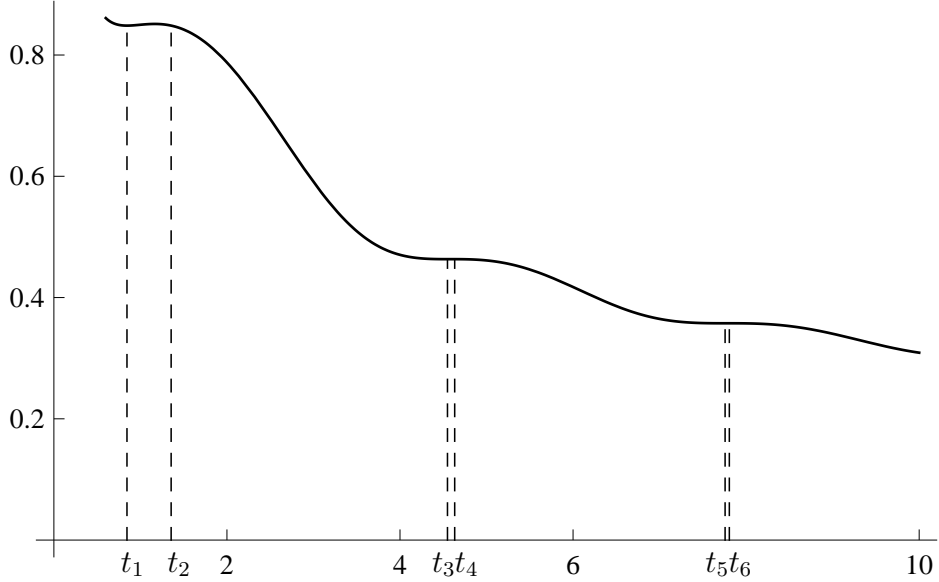}}
\put(0.12,0.01){$t_1$}
\put(0.17,0.01){$t_2$}
\put(0.455,0.01){$t_3$}
\put(0.48,0.01){$t_4$}
\put(0.75,0.01){$t_5$}
\put(0.775,0.01){$t_6$}
\end{picture}
\end{center}
\caption{Function $r(t)$ for $p(t)=t^{-1/2}$, see Lemma~\ref{tehnicka2}, $t_0=0.6$. This is a wavy function, see Definition~\ref{def_wavy_function}, with local minima at $t_{2k+1}$, $k=0,1,\ldots$}
\label{fig:funkcija_r}
\end{figure}

\setlength{\unitlength}{0.7\textwidth}
\begin{figure}
\begin{center}
\begin{picture}(1,0.82)
\put(0,0){\includegraphics[width=0.7\textwidth]{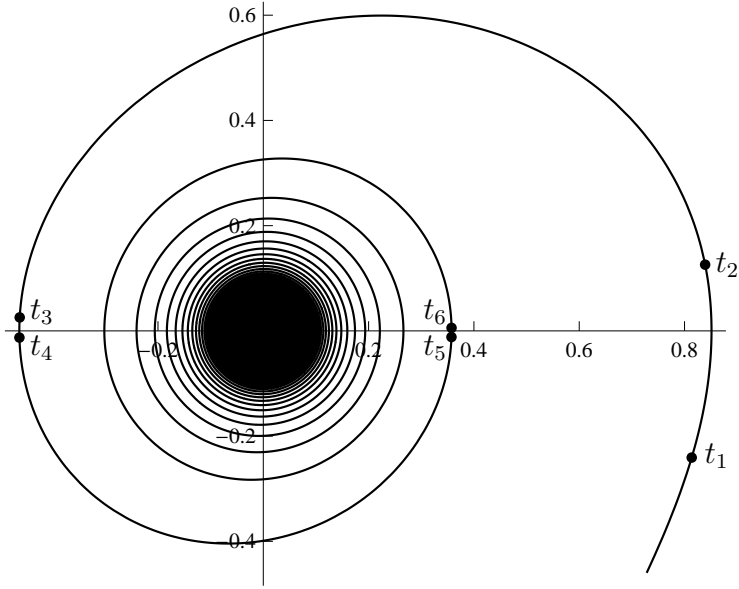}}
\put(0.97,0.17){$t_1$}
\put(0.985,0.435){$t_2$}
\put(0.035,0.37){$t_3$}
\put(0.035,0.32){$t_4$}
\put(0.58,0.32){$t_5$}
\put(0.58,0.37){$t_6$}
\end{picture}
\end{center}
\caption{Spiral $\Gamma'$ for $p(t)=t^{-1/2}$, see Lemma~\ref{tehnicka2}, $t_0=0.6$ is a wavy spiral, see Definition~ \ref{def_wavy_spiral}. Highlighted points correspond to parameters $t_k$, $k=1,2,\ldots$ }
\label{fig:spirala}
\end{figure}

\begin{definition}\label{def_wavy_spiral} Let a spiral $\Gamma'$, given in polar coordinates by $r=f(\f)$, where $f$ is a given function. If there exists increasing or decreasing function of class $C^1$, $\f=\f(t)$, such that $r(t)=f(\f(t))$ is a wavy function, then we say $\Gamma'$ is a \emph{wavy spiral}.
\end{definition}

For an example of spiral $\Gamma'$, see Figure \ref{fig:spirala}. Now, using Theorem \ref{novizuzu} and Lemma \ref{novalemazuzu} we prove the following Theorem \ref{biliplema}.

\begin{theorem}[Box dimension of a wavy spiral]{\label{biliplema}}
Let $t_0>0$ and assume $r:[t_0,\infty)\rightarrow(0,\infty)$ is a wavy function. Assume that $\f:[t_0,\infty)\rightarrow[\f_0,\infty)$ is an increasing function of class $C^1$ such that $\f(t_0)=\f_0>0$ and there exists $\bar{\f}_0\in\eR$ such that
\begin{equation} \label{uvjetfikaot}
|(\f(t)-\bar{\f}_0)-(t-t_0)|\to 0\ \ \textrm{as}\ \ t\to\infty.
\end{equation}
Let $f:[\f_0,\infty)\rightarrow(0,\infty)$ be defined by $f(\f(t))=r(t)$. Assume that $\Gamma'$ is a spiral defined in polar coordinates by $r=f(\f)$, satisfying $(\ref{def_spiral})$.
Let $\alpha\in(0,1)$ is the same value as in $(\ref{def_SRC})${\rm(iii)} for wavy function $r$, and assume $\varepsilon'$ is such that $0<\varepsilon'<\varepsilon$, where $\varepsilon$ is defined by $(\ref{def_SRC})${\rm(ii)} for wavy function $r$. Assume that there exist positive constants $\underline{m}$, $\overline{m}$, $\underline{a}'$ and $M$ such that for all $\f\geq\f_0$,
\begin{equation}\label{biliplema_cond1}
\underline{m}\f^{-\alpha}\leq f(\f)\leq \overline{m}\f^{-\alpha} ,
\end{equation}
\begin{equation}\label{biliplema_cond3}
|f'(\f)|\leq M\f^{-\alpha-1} ,
\end{equation}
and for all $\triangle\f$, such that $\theta\leq\triangle\f\leq 2\pi+\theta$, there holds
\begin{equation}\label{biliplema_cond2b}
\underline{a}'\f^{-\alpha-1}\leq f(\f)-f(\f+\triangle\f) ,
\end{equation}
where $\theta:=\min\left\{\varepsilon',\pi\right\}$.

Then $\Gamma'$ is a wavy spiral and
$$
\dim_B\Gamma'=\frac{2}{1+\alpha}.
$$
\end{theorem}

The proof of Theorem \ref{biliplema} is given in \cite{bessel}.

\smallskip

Now, Theorem \ref{biliplema} enables us to calculate the box dimension of a spiral generated by a chirp, which is one of the main results of this paper.

\begin{theorem}[Chirp--spiral comparison]{\label{s-c-p}}
Let $\alpha>0$. Assume that  $X:(0,1/\tau_0]\to\eR$, $\tau_0>0$, $X(\tau)=P(\tau)\sin 1/\tau$, where $P(\tau)$ is a positive function such that $P(\tau)\sim_3\tau^\alpha$ as $\tau\to 0$.
Define $x(t):=X(1/t)$ and a continuous function $\varphi (t)$ by
$\tan \f(t)=\frac{\dot x(t)}{x(t)}$.

\begin{itemize}
\item[\rm{(i)}] If $\alpha\in(0,1)$ then the planar curve $\Gamma:=\{(x(t),\dot x(t))\in\eR: t\in[\tau_0,\infty)\}$ generated by $X$ is a wavy spiral $r=f(\f)$, $\f\in(-\infty,-\phi_0]$ near the origin. We have $f(\f)\simeq|\f|^{-\alpha}$ as $\f\to-\infty$, and
    $$
    \dim_{ph}(x):=\dim_B\Gamma=\frac{2}{1+\alpha} .
    $$
\item[\rm{(ii)}] If $\alpha>1$ then the planar curve $\Gamma:=\{(x(t),\dot x(t))\in\eR: t\in[\tau_0,\infty)\}$ is a rectifiable wavy spiral near the origin.
\end{itemize}

\end{theorem}

The proof of Theorem~\ref{s-c-p} consists of checking out the conditions of Theorem~\ref{biliplema}. The following lemmas make this checking easy.

\begin{lemma}\label{pp} Let $\alpha>0$ and assume that $P(\tau)$, $\tau\in(0,1/t_0]$, $t_0>0$, be such that $P(\tau)\sim_3\tau^\alpha$ as $\tau\to 0$. Then
$p(t):=P(\frac{1}{t})\sim_3t^{-\alpha}$ as $t\to\infty$ and vice versa. Furthermore, we have:
\begin{equation}\label{plim}
 \lim_{t\to \infty} \frac{p'(t)}{p(t)}=0, \ \lim_{t\to \infty}\frac{p''(t)}{p(t)}=0,
\end{equation}
\begin{equation}\label{pnej}
 -\frac{p(t)}{p'(t)}\sim \frac{t}{\alpha}, \ \ -\frac{2p'(t)}{p''(t)}\sim \frac{2t}{\alpha+1}\ \mathrm{as}\ t\to\infty,
\end{equation}
\begin{equation}\label{psup}
 \sup_{t\in[t_0,\infty)}\left( -\frac{p(t)}{p'(t)}\right)'<\infty, \ \sup_{t\in[t_0,\infty)}\left( -\frac{2p'(t)}{p''(t)}\right)'<\infty.
\end{equation}
\end{lemma}

The claims of Lemma \ref{pp} follow directly from the assumptions.

\begin{lemma}{\label{tehnicka2}}
Let $\alpha\in(0,1)$ and
$$
r(t)=p(t)\sqrt{1+\frac{p'^2(t)}{p^2(t)} \sin^ 2t+\frac{p'(t)}{p(t)}\sin 2t},\quad t\in[t_0,\infty),\ t_0>0,
$$
where $p(t)\sim_1t^{-\alpha}$ as $t\to\infty$.

Let $C\in\eR$ and assume that $t(\f)=\f+C+O(\f^{-1})$ as $\f\to\infty$. Let $\triangle\f>1$. Then there exists constant $k>0$, independent of $\f$ and $\triangle\f$, such that for all $\f$ sufficiently large it holds
$$
r(t(\f))-r(t(\f+\triangle\f)) \geq k\f^{-\alpha-1}(1+O(\f^{-1})).
$$
\end{lemma}

The proof of Lemma \ref{tehnicka2} can be seen in the Appendix.

\begin{proof}[Theorem \ref{s-c-p}]
(i) \emph{Step 1.} (The box dimension is invariant with respect to mirroring of a spiral.) We will prove the equivalent claim, that planar curve $\Gamma'=\{(x(t),-\dot x(t)) : t\in[\tau_0,\infty)\}$ is a wavy spiral defined by $r=f(\f)$, $\f\in[\phi_0, \infty)$, near the origin, satisfying $f(\f)\simeq\f^{-\alpha}$, in polar coordinates, near the origin, and $\dim_B\Gamma'=\frac{2}{1+\alpha}$. It is easy to see that curve $\Gamma$ is a mirror image of curve $\Gamma'$, with respect to the $x$-axis, hence $\Gamma$ is a wavy spiral. Reflecting with respect to the $x$-axis in
the plane is an isometric map. As the isometric map is bi-Lipschitzian and therefore it preserves box dimensions, see \cite[p. 44]{falc}, we see that $\dim_B\Gamma = \dim_B\Gamma'=\frac{2}{1+\alpha}$.

\smallskip

\emph{Step 2.} (Checking condition (\ref{biliplema_cond1}).) From
\begin{eqnarray}
 x(t)&=&p(t) \sin t , \nonumber  \\
 \dot x (t)&=&  p'(t)\sin t +p(t) \cos t , \nonumber
\end{eqnarray}
where $p(t):=P(1/t)$, we compute
\begin{equation}\label{fip}
\tan \f(t)=-\frac{\dot x(t)}{x(t)}=-\frac{p'(t)}{p(t)} -\frac{1}{\tan t}.
\end{equation}
By differentiating (\ref{fip}) we obtain
\begin{equation}\label{dfipodtp}
\frac{d\varphi}{dt}(t)=\cos ^2\varphi(t) \left[\frac{p'^{2}(t)-p(t)p''(t)}{p^2(t)}+\frac{1}{\sin ^2 t}\right].
\end{equation}
Using (\ref{fip}) again, we have
\begin{equation} \label{coskvad_fiodt}
\cos^2 \varphi(t)=\frac{1}{1+\tan^2 \varphi(t)}=\frac{p^2(t)\sin^2 t}{p^2(t) + p'^2(t)\sin^2 t + 2p(t)p'(t)\sin t \cos t} .
\end{equation}
Substituting into (\ref{dfipodtp}) and using (\ref{plim}) we get
\begin{equation}\label{dfi}
\lim_{t\to\infty}\frac{d\varphi}{dt}(t)=1.
\end{equation}
From (\ref{dfi}) it follows that $\f \simeq t\ \mathrm{as}\ t\to\infty$ and
\begin{equation}\label{rkvp}
r^2(t)=(x(t))^2+(-\dot x(t))^2 = p^2(t)+p'^2(t)\sin ^2 t+p(t)p'(t)\sin 2t
\end{equation}
implies
\begin{equation}\label{fkaotkaofi}
f(\f(t)) = r(t)\simeq t^{-\alpha}\simeq \f^{-\alpha}\ \mathrm{as}\ t\to\infty .
\end{equation}
Notice that from (\ref{rkvp}) it follows that function $r(t)$ is of class $C^2$ and by substituting (\ref{coskvad_fiodt}) in (\ref{dfipodtp}), respecting (\ref{rkvp}), we see that function $\f(t)$ is of class~$C^1$.

\smallskip

\emph{Step 3.} (Checking condition (\ref{biliplema_cond3}).) On the other hand, differentiating (\ref{rkvp}) we obtain that
\begin{eqnarray} \label{drpodtp}
\frac{dr}{dt}(t)=\left[2p(t)p'(t)\cos^2 t \right. & + & \frac{2p'^2(t)+p(t)p''(t)}{2}\sin 2t \\
  & + & \left. p'(t)p''(t)\sin^2 t\right]\frac{1}{r(t)} . \nonumber
\end{eqnarray}
Also, from (\ref{drpodtp}) we have
\begin{equation} \label{drpodtp'}
\frac{dr}{dt}(t)=\frac{2p(t)p'(t)}{r(t)}\cos^2 t + O(t^{-\alpha-2})\ \mathrm{as}\ t\to\infty.
\end{equation}
Since $\frac{dr}{dt}(t)=f'(\f)\cdot\frac{d\f}{dt}(t)$ and since by (\ref{dfi}) we have  $\frac{d\f}{dt}(t)\simeq 1$ as $t\to\infty$, there exists $C_0>0$ and $C_1>C_0$ such that
$$
|f'(\varphi)|\le C_0 t^{-\alpha -1}\le C_1 \f^{-\alpha-1}\ \textrm{as}\ \f\to\infty .
$$

\smallskip

\emph{Step 4.} (Checking condition (\ref{uvjetfikaot}).) Using (\ref{fip}) and \cite[Lemma 7]{bessel}, we obtain
$$
\tan\f(t)=-(\cot t+O(t^{-1}))=-\cot(t+O(t^{-1}))=\tan(t+\frac{\pi}{2}+O(t^{-1}))
$$
as $t\to\infty$. Since function $\f(t)$ is continuous by the definition and $O(t^{-1})<\pi$ for $t$ large enough, then there exists $k\in\mathbb{Z}$ such that
$$
\f(t)=(t+\frac{\pi}{2}+k\pi)+O(t^{-1})\ \mathrm{ as\ } t\to\infty.
$$
From the definition of $\f(t)$ we conclude that  we may take without loss of generality $k=0$. Finally, we get
\begin{equation} \label{fiprekot}
\f(t)=\left(t+\frac{\pi}{2}\right)+O(t^{-1})\ \mathrm{ as\ } t\to\infty.
\end{equation}

\smallskip

\emph{Step 5.} (Checking condition (\ref{biliplema_cond2b}).) From (\ref{dfi}) it follows that there exists $\tau_1\geq \tau_0$ such that $\frac{d\f}{dt}(t)>0$ for all $t\geq\tau_1$ hence the function $\f(t)$ is increasing for all $t$ sufficiently large. As function $\f(t)$ is continuous, we conclude that for all $\f$ large enough there exists inverse function $t=t(\f)$ of function $\f=\f(t)$ and
$$
t(\f)= \left(\f-\frac{\pi}{2}\right)+O(\f^{-1})\ \mathrm{as\ } \f\to\infty.
$$
Define value $\phi_1:=\f(\tau_1)$ and notice that we can take $\tau_1$ sufficiently large such that $\phi_1\geq\phi_0$.

From (\ref{rkvp}) we obtain
$$
r(t)=p(t)\sqrt{1+\frac{p'^2(t)}{p^2(t)} \sin^ 2t+\frac{p'(t)}{p(t)}\sin 2t}.
$$

By Lemma \ref{tehnicka2} we conclude that for fixed $\triangle\f>1$ we have
\begin{equation}\label{fiidelfi}
f(\f)-f(\f+\triangle\f)=r(t(\f))-r(t(\f+\triangle\f))\geq k_1\f^{-\alpha-1},
\end{equation}
provided $\f$ is large enough. Moreover, by careful examination of the proof of Lemma~\ref{tehnicka2}, we conclude that statement (\ref{fiidelfi}) uniformly holds for every $\triangle\f$ from a bounded interval whose lower bound is greater than $1$, also provided $\f$ is large enough. (Notice we will have to take $\theta$ from Theorem \ref{biliplema} to be larger than $1$.)

\smallskip

\emph{Step 6.} ($\Gamma'$ is a spiral near the origin.) Now we can prove that $\Gamma'$ is a spiral near the origin, that is, $f(\f)$ satisfies condition (\ref{def_spiral}) near the origin. First, from (\ref{fkaotkaofi}) it follows that $f(\f)\to 0$ as $\f\to\infty$. Second, from (\ref{fiidelfi}) it follows that $f(\f)$ is radially decreasing for all $\f$ large enough, that is, there exists $\phi_2\geq\phi_1$ such that $f|_{[\phi_2,\infty)}$ is radially decreasing.

\smallskip

\emph{Step 7.} (The box dimension is invariant with respect to taking $\tau_0$ and $\phi_0$ sufficiently large.) First, we define $\tau_2$ to be such that $\f(\tau_2)=\phi_2$. Notice that $\tau_2$ is well-defined and $\tau_2\geq \tau_1$. As $p(t)>0$, from (\ref{rkvp}) and the definition of $x(t)$ and $\dot{x}(t)$ it follows that $r(t)>0$, that is, $r(t)$ is strictly positive function. That means there exists constant $m_1>0$ such that for all $t\in[\tau_0,\tau_2]$ it holds
\begin{equation}\label{tvrdnja_vece}
r(t)>m_1 .
\end{equation}
Notice that $\phi_2\geq\phi_1\geq\phi_0$. From (\ref{fkaotkaofi}) it follows that $r(t)\to 0$ as $t\to\infty$ so there exists $\tau_3\geq\tau_2$ such that for all $t\in[\tau_3,\infty)$ it holds
\begin{equation}\label{tvrdnja_manje}
r(t)<m_1 .
\end{equation}
We define $\phi_3:=\f(\tau_3)$. Notice that we could increase $\tau_3$ and $\phi_3$ to accommodate all requirements in different parts of the proof on $t$ or $\f$ being sufficiently large. Now, from (\ref{tvrdnja_vece}) and (\ref{tvrdnja_manje}) we conclude that
\begin{equation}\label{komad_jedan}
\Gamma'|_{[\tau_0,\tau_2]}\bigcap\Gamma'|_{(\tau_3,\infty)}=\emptyset .
\end{equation}
As $f|_{[\phi_2,\infty)}$ is radially decreasing and $\f'(t)>0$ for all $t\in[\tau_2,\infty)$ it follows that $\Gamma'|_{(\tau_2,\infty)}$ does not have self intersections, so
\begin{equation}\label{komad_dva}
\Gamma'|_{[\tau_2,\tau_3]}\bigcap\Gamma'|_{(\tau_3,\infty)}=\emptyset .
\end{equation}
Finally, from (\ref{komad_jedan}) and (\ref{komad_dva}) we have $\Gamma'|_{[\tau_0,\tau_3]}\bigcap\Gamma'|_{(\tau_3,\infty)}=\emptyset$. Now, we can apply Lemma \ref{novalemazuzu} on curve $\Gamma'$.

Using Lemma \ref{novalemazuzu} we see that without loss of generality we can assume that $\tau_0$ and $\phi_0$, in the assumptions of the theorem, are sufficiently large. Informally, we can always remove any rectifiable part from the beginning of $\Gamma'$, without changing the box dimension of $\Gamma'$.

\smallskip

\emph{Step 8.} (Checking waviness condition (\ref{def_SRC}).) By factoring (\ref{drpodtp}), we get
\begin{equation}\label{drpodtfactp}
\frac{dr}{dt}(t)=\left(1+\frac{p'(t)}{p(t)}\tan t\right)\left(1+\frac{p''(t)}{2p'(t)}\tan t\right)\frac{2p(t)p'(t)}{r(t)}\cos^2 t ,
\end{equation}
for every $t\neq\frac{\pi}{2}+k\pi$, $k\in\Ze$ ($\cos t\neq 0$).

By Lemma \ref{unique1} and Remark \ref{remark-za-drpodt}, see below, and using (\ref{pnej}) and (\ref{psup}), there exists $k_0\in\eN_0$ such that the equations
$$
       \tan t=-\frac{p(t)}{p'(t)}, \quad \quad \tan t =-\frac{2p'(t)}{p''(t)},
$$
have unique solutions $\hat{t}_{2k}$ and $t_{2k-1}$, respectively, in intervals $((k+k_0)\pi-\pi,(k+k_0)\pi-\frac{\pi}{2})$, for each $k\in\eN_0$, since
$$
-\frac{p(t)}{p'(t)}\sim \frac{t}{\alpha}, \qquad -\frac{2p'(t)}{p''(t)}\sim \frac{2t}{\alpha+1} \qquad\mathrm{as}\ t\to\infty .
$$
Moreover, by taking $k_0$ to be sufficiently large, from (\ref{pnej}) and using inequalities $1<2/(\alpha+1)<1/\alpha$, we see that $\hat{t}_{2k}$ and $t_{2k-1}$ even lie in the smaller intervals
\begin{equation} \label{interval_za_poz_t}
((k+k_0)\pi-\frac{\pi}{2}-\frac{\pi}{3},(k+k_0)\pi-\frac{\pi}{2}) ,
\end{equation}
for each $k\in\eN_0$. (The statement is true for interval of any length provided upper bound is $(k+k_0)\pi-\frac{\pi}{2}$. We choose value $\pi/3$, because it is convenient later in the proof.)

Notice that because of $\frac{1}{\alpha}\neq\frac{2}{\alpha+1}$ we see that $-\frac{p(t)}{p'(t)}\neq -\frac{2p'(t)}{p''(t)}$ for $t$ sufficiently large, so $\hat{t}_{2k}\neq t_{2k-1}$ for $k_0$ sufficiently large. Without loss of generality we can take $t_{2k-1}<\hat{t}_{2k}$. So $\hat{t}_{2k}-t_{2k-1}<\pi/3$ for every $k\in\eN$, provided $k_0$ is sufficiently large. It is easy to see from (\ref{drpodtfactp}) that $\frac{dr}{dt}(t)$ is positive between these solutions.

As $\frac{d\f}{dt}(t)>0$ for all $t$ sufficiently large, from $\frac{dr}{dt}(t)=f'(\f)\cdot\frac{d\f}{dt}(t)$ it follows that $f'(\f)>0$ on set $\bigcup_{k=1}^{\infty}(\f_{2k-1},\hat{\f}_{2k})$ where $\f_{2k-1}:=\f(t_{2k-1})$ and $\hat{\f}_{2k}:=\f(\hat{t}_{2k})$.
This implies that function $f(\f)$ is increasing for some $\f$, so we can not apply Theorem \ref{novizuzu} directly.

Notice that if $t\in\bigcup_{k=0}^{\infty}(t_{2k-1},\hat{t}_{2k})$ then $r'(t)>0$ and if $t\in\bigcup_{k=0}^{\infty}(\hat{t}_{2k},t_{2k+1})$ then $r'(t)<0$.

We would like to prove that for every $k\in\eN_0$ there exists unique $t_{2k}\in(\hat{t}_{2k},t_{2k+1})$ such that $r(t_{2k})=r(t_{2k-1})$ and $t_{2k}-t_{2k-1}< \pi/3$ (where we will take $k_0$ from (\ref{interval_za_poz_t}) to be sufficiently large). As $r(\hat{t}_{2k})>r(t_{2k-1})$, and as function $r(t)$ is continuous and strictly decreasing on interval $(\hat{t}_{2k},t_{2k+1})$, it follows that, if such $t_{2k}$ exists then it is necessary unique, so we only need to prove the existence.

For every $k\in\eN_0$ we take $\bar{t}_{2k}:=t_{2k-1}+\pi/3$. Notice that $\bar{t}_{2k}\in(\hat{t}_{2k},t_{2k+1})$, because from (\ref{interval_za_poz_t}) follows that $t_{2k+1}-t_{2k-1}>2\pi/3$ and $\hat{t}_{2k}-t_{2k-1}<\pi/3$. Define $\bar{\f}_{2k}:=\f(\bar{t}_{2k})$ and take $\f_{2k-1}$ as defined before. Using (\ref{fiprekot}), we can take $t$ or equivalently $k_0$ sufficiently large, such that $(\pi/3+1)/2\leq\bar{\f}_{2k}-\f_{2k-1}\leq 2$ for every $k\in\eN_0$. (The exact value of the upper bound is not important. We just take some value larger than $\pi/3$. For lower bound, it is only important that it is larger than $1$ and lower than $\pi/3$, so we take the mean value between these two.)

Now, using Lemma~\ref{tehnicka2}, analogously as in \emph{Step 5}, we compute
\begin{eqnarray}
r(t_{2k-1})-r(\bar{t}_{2k}) & = & r(t(\f_{2k-1}))-r(t(\bar{\f}_{2k})) \nonumber\\
& = & r(t(\f_{2k-1}))-r(t(\f_{2k-1}+(\bar{\f}_{2k}-\f_{2k-1}))) \nonumber\\
& \geq & C_2 \f_{2k-1}^{-\alpha-1} > 0 , \nonumber
\end{eqnarray}
for some $C_2>0$, provided $\f$ or equivalently $k_0$ is sufficiently large. From this follows $r(\bar{t}_{2k})<r(t_{2k-1})$, and as function $r(t)$ is of class $C^1$, strictly decreasing on interval $(\hat{t}_{2k},\bar{t}_{2k})$ and $r(\hat{t}_{2k})>r(t_{2k-1})$, we see that there exist $t_{2k}\in(\hat{t}_{2k},\bar{t}_{2k})$ such that $r(t_{2k})=r(t_{2k-1})$ and obviously $t_{2k}-t_{2k-1}< \pi/3$.

Using $t_{2k+1}-t_{2k-1}>2\pi/3$, follows that $t_{2k+1}-t_{2k}>2\pi/3-\pi/3=\pi/3$.

We established that for every $k\in\eN_0$ holds $t_{2k+1}>t_{2k}>t_{2k-1}$. Notice that $r'(t_0)\leq 0$ and that sequence $(t_n)$, $n\in\eN_0$, is the same as the sequence from Definition \ref{def_wavy_function}, defined for function $r(t)$.

As $t_{2k+1}-t_{2k-1}>2\pi/3$ for every $k\in\eN_0$, we conclude that $t_n\to\infty$ as $n\to\infty$, which means that sequence $(t_n)$ satisfies condition (\ref{def_SRC})(i).

As $t_{2k+1}-t_{2k}>\pi/3$ for every $k\in\eN_0$, by taking $\varepsilon=\pi/3$, we see that sequence $(t_n)$ satisfies condition (\ref{def_SRC})(ii).

Using (\ref{drpodtp'}) we conclude that there exist $C_3, C_4\in\eR$, $C_4>C_3>0$, such that
\begin{eqnarray}
\mathop{\mathrm{osc}}\limits_{t\in[t_{2k+1},t_{2k+2}]} r(t) & = & r(\hat{t}_{2k+2}) - r(t_{2k+1}) = \int\limits_{t_{2k+1}}^{\hat{t}_{2k+2}} r'(t)\,dt \nonumber\\
& \leq & \frac{1}{3}\cdot\sup\limits_{t\in[t_{2k+1},\hat{t}_{2k+2}]} r'(t) \leq C_3 t_{2k+1}^{-\alpha-2} \leq C_4\hat{t}_{2k+2}^{-\alpha-2} , \nonumber
\end{eqnarray}
for every $k\in\eN_0$, which means that sequence $(t_n)$ satisfies condition (\ref{def_SRC})(iii).

Finally, we conclude that sequence $(t_n)$ satisfies waviness condition (\ref{def_SRC}), so $r(t)$ is a wavy function and $\Gamma'$ is a wavy spiral near the origin.

\smallskip

\emph{Step 9.} (Final conclusion.) From the previous steps we see directly that all of the assumptions of Theorem \ref{biliplema} are fulfilled. We take $\varepsilon'=(\pi/3+1)/2<\varepsilon$ and $\theta=\min\{\varepsilon',\pi\}=(\pi/3+1)/2$. Using Theorem \ref{biliplema} we prove that
$\dim_B\Gamma'=\frac{2}{1+\alpha}$.

\medskip

(ii) To prove that $\Gamma$ is a wavy spiral near the origin notice that \emph{Steps 1--8} also hold for $\alpha>1$.

To prove the rectifiability for $\alpha>1$, from (\ref{fkaotkaofi}), (\ref{dfi}) and (\ref{drpodtp'}) we have that there exist positive constants $C_5$, $M_1$ and $C_6$ such that for every $t\in[t_0,\infty)$ it holds
$$
r(t)\leq C_5 t^{-\alpha},\quad \f'(t)\leq M_1,\quad |r'(t)|\leq C_6 t^{-\alpha-1} .
$$
Therefore
\begin{eqnarray}
l(\Gamma)& = & l(\Gamma') = \int_{t_0}^{\infty}\sqrt{(r(t)\f'(t))^2+(r'(t))^2}\,dt  \nonumber\\
& \leq & \int_{t_0}^{\infty}\sqrt{ M_1^2 C_5^2 t^{-2\alpha}+C_6^2 t^{-2\alpha-2}}\,dt \leq M_2(t_0)\int_{t_0}^{\infty} |t|^{-\alpha}\,dt<\infty . \nonumber
\end{eqnarray}
\qed
\end{proof}

\section{Chirps generated by spirals}

Now we study some a converse of Theorem \ref{s-c-p}, where we obtain the box dimension of a chirp from the corresponding spiral. We begin with a theorem concerning the box dimension of the graph of a generalized $(\alpha,\beta)$-chirp.

\begin{theorem}\label{BDchirp} {\bf (Box dimension and Minkowski content of the graph of a generalized $(\alpha,\beta)$-chirp)}
Let $y(x)=p(x)S(q(x))$, $x \in I=(0,c],$ $c>0.$
Let the functions $p(x)$, $q(x)$ and $S(t)$ satisfy the following assumptions:
\begin{equation}\label{px}
\mbox{ $p\in C(\bar{I})\cap C^{1}(I)$, $q\in C^{1}(I)$, $S\in C^1(\eR)$},
\end{equation}

The function $S(t)$ is assumed to be a $2T$-periodic real function defined on  $\eR$  such that

\begin{equation}\label{Sx1}
\left\{
\begin{array}{c}
\mbox{$S(a)=S(a+T)=0$ for some $a\in\eR$,}\\
\mbox{$S(t)\neq 0$ for all $t\in (a,a+T)\cup (a+T,a+2T)$,}
\end{array}
\right.
\end{equation}
where $T$ is a positive real number and  $S(t)$ alternately changes sign on intervals $(a+(k-1)T,a+kT)$, $k\in \eN$. Without loss of generality, we take  $a=0$.
Let us suppose that $0<\alpha \leq \beta$ and:
\begin{equation}\label{apsp}
p(x)\simeq_1 x^{\alpha} \quad \mbox{as} \quad x\to 0,
\end{equation}
\begin{equation}\label{qq}
q(x)\simeq_1 x^{-\beta} \quad \mbox{as} \quad x\to 0.
\end{equation}

Then
$y(x)$ is d-dimensional fractal oscillatory near the origin, where $d=2-\frac{\alpha+1}{\beta+1}$.
Moreover, $\dim_B(G(y))=d$ and $G(y)$ is Minkowski nondegenerate.
\end{theorem}

Theorem \ref{BDchirp} is a new version of \cite[Theorems 5 and 6]{lukamaja}. In Theorem~\ref{BDchirp} we do not need any assumptions on the curvature function of $y(x)=p(x)S(q(x))$, as it was needed in \cite{lukamaja}.
Before proving Theorem \ref{BDchirp}, we shall cite a new criterion for fractal oscillations of a bounded continuous function and after that we continue with two propositions concerning the properties of functions $p, q$ and $S$.

\begin{theorem}[Theorem 2.1. from \cite{mersat}]{\label{mersat}} Let $y\in C^1((0,T])$ be a bounded function on $(0,T]$. Let $s\in[1,2)$ be a real number and let $(a_n)$ be a decreasing
sequence of consecutive zeros of $y(x)$ in $(0,T]$ such that $a_n\to 0$ when $n\to \infty$ and let there exist constants $c_1, c_2, \varepsilon_0$ such that for all $\varepsilon \in(0,\varepsilon_0)$ we have:
\begin{equation}\label{left}
c_1\varepsilon^{2-s}\leq \sum_{n\geq k(\varepsilon)} \max_{x\in[a_{n+1},a_n]} |y(x)|(a_n-a_{n+1}),
\end{equation}
\begin{equation}\label{right}
a_{k(\varepsilon)} \sup_{x\in(0,a_{k(\varepsilon)}]} |y(x)| + \varepsilon \int_{a_{k(\varepsilon)}}^{a_1} |y'(x)|dx \leq c_2\varepsilon^{2-s},
\end{equation}
where $k(\varepsilon)$ is an index function on $(0,\varepsilon_0]$ such that
$$
|a_n-a_{n+1}|\leq \varepsilon \quad \mbox{for all} \quad n\geq k(\varepsilon) \quad\mbox{and}\quad \varepsilon \in (0,\varepsilon_0).
$$
Then $y(x)$ is fractal oscillatory near $x=0$ with $\dim_B G(y)=s$.
\end{theorem}

We remark that the claim of Theorem \ref{mersat} is true if we substitute $a_1$, appearing in (\ref{right}) by $a_{k_0}$, where $k_0$ is a fixed natural number.

\begin{proposition}{\label{druga}} Assume that the functions $p(x)$ and  $q(x)$ satisfy conditions {\rm (\ref{px})}, {\rm (\ref{apsp})} and {\rm (\ref{qq})}. Then
there exist $\delta_0 >0$ and positive constants $C_1 \mbox{and} \  C_2$ such that:
               \begin{equation}\label{qq1}
               C_1x^{\alpha}\leq p(x) \leq C_2x^{\alpha}, \  C_1x^{\alpha-1}\leq p'(x) \leq C_2x^{\alpha-1},
               \end{equation}
               \begin{equation}\label{qq2}
                C_1x^{-\beta}\leq q(x) \leq C_2x^{-\beta}, \  C_1x^{-\beta-1}\leq -q'(x) \leq C_2x^{-\beta-1},
               \end{equation}
                for all $x\in(0,\delta_0]$. Furthermore, there exists the inverse function $q^{-1}$ of the function $q$ defined on $[m_0,\infty)$, where $m_0=q(\delta_0)$, and it holds:
                \begin{equation}\label{qq-1}
               q^{-1}(t) \simeq_1 t^{-1/\beta}\quad \mbox{as}\quad t\to\infty,
               \end{equation}
               \begin{equation}\label{qq-1st}
               C_1t^{-\frac{1}{\beta}-1}(t-s)\leq q^{-1}(s)- q^{-1}(t)\leq C_2s^{-\frac{1}{\beta}-1}(t-s), \quad m_0\leq s<t.
               \end{equation}
\end{proposition}

\begin{proposition}{\label{treca}} For any function $S(t)$ satisfying {\rm (\ref{Sx1})}, and for any function $q(x)$ with properties {\rm (\ref{px})} and {\rm (\ref{qq})}, we have:
\begin{itemize}\item[\rm{(i)}]
$S(kT)=0, \ k\in \eN$.
                \item[\rm{(ii)}]
Let $a_k=q^{-1}(kT)$ and $s_k=q^{-1}(t_0+kT), \ k\in \eN$, where $t_0\in(0,T)$ is arbitrary. Then there exist $k_0\in \eN$ and $c_0>0$ such that $a_k\in(0,\delta_0]\ y(a_k)=0,\
 s_k\in(a_{k+1},a_k)$ for all $k\geq k_0$, $a_k\searrow 0$ as $k\to \infty$, $a_k\simeq k^{-1/\beta}$ as $k\to \infty$, and
\begin{equation}\label{pro4}
 \max_{x\in[a_{k+1},a_k]} |y(x)| \geq c_0(k+1)^{-\alpha/\beta} \quad \mbox{for all $k\geq k_0$, $c_0>0$}.
\end{equation}
                \item[\rm{(iii)}]
There exists $\varepsilon_0>0$ and a function $k:(0,\varepsilon_0)\to\eN$ such that
\begin{equation}\label{kaeps}
\frac{1}{T}\left(\frac{\varepsilon}{TC_2} \right)^{-\frac{\beta}{\beta+1}}\leq k(\varepsilon)\leq \frac{2}{T}\left(\frac{\varepsilon}{TC_2} \right)^{-\frac{\beta}{\beta+1}}.
\end{equation}
In particular,
$$C_1T((k+1)T)^{-\frac{1}{\beta}-1}\leq a_k-a_{k+1}\leq \varepsilon ,$$
for all $k\geq k(\varepsilon)$ and $\varepsilon\in (0,\varepsilon_0)$.
\end{itemize}
\end{proposition}

Proofs of Propositions \ref{druga} and \ref{treca} are provided in the Appendix.

\smallskip

\begin{proof}[Theorem \ref{BDchirp}]
First we check inequality (\ref{left}). By Proposition \ref{treca} we have:
\begin{eqnarray}
\sum_{k\geq k(\varepsilon)} \max_{x\in[a_{k+1},a_k]} |y(x)|(a_k-a_{k+1}) & \geq & c\sum_{k=k(\varepsilon)+1}^\infty (k+1)^{-\frac{\alpha+\beta+1}{\beta}} \nonumber\\
& = & c\sum_{k=k(\varepsilon)}^\infty (k)^{-\frac{\alpha+\beta+1}{\beta}}=ca, \nonumber
\end{eqnarray}
where the series $ a=\sum_{k=k(\varepsilon)}^\infty (k)^{-\frac{\alpha+\beta+1}{\beta}}$ is convergent, because of $\frac{\alpha+\beta+1}{\beta}>1$. Then using the inequality
$(\frac{1}{k(\varepsilon)})^{\frac{\alpha+\beta+1}{\beta}-1}<1$ and (\ref{kaeps}) we obtain
\begin{eqnarray}
\sum_{k\geq k(\varepsilon)} \max_{x\in[a_{k+1},a_k]} |y(x)|(a_k-a_{k+1})&\geq& ca\geq ca(\frac{1}{k(\varepsilon)})^{\frac{\alpha+\beta+1}{\beta}-1}\nonumber\\
&\geq& c_1\varepsilon^{\frac{\alpha+1}{\beta+1}}=c_1\varepsilon^{2-(2-\frac{\alpha+1}{\beta+1})},\nonumber
\end{eqnarray}
for all $\varepsilon \in (0,\varepsilon_0)$. By \cite[Lemma 2.1.]{mersat} this implies
$$0<\M_*^d(G(y)) \quad \mbox{and} \quad \underline\dim_BG(y)   \geq d,
$$
where $G(y)$ is the graph of $y$ and $d=2-\frac{\alpha+1}{\beta+1}$. Now we check inequality (\ref{right}). From (\ref{apsp}) and (\ref{qq}) it follows that
$$|y'(x)|=|p'(x)S(q(x))+p(x)q'(x)S'(q(x))|\leq cx^{\alpha-\beta-1},$$
which holds near $x=0$, where $$c=\max\{ \max_{x\in[0,2T]}|S(t)|,\max_{x\in[0,2T]}|S'(t)| \}.$$ By Proposition \ref{treca} we have:
$$
a_{k(\varepsilon)} \sup_{x\in(0,a_{k(\varepsilon)}]} |y(x)| + \varepsilon \int_{a_{k(\varepsilon)}}^{a_{k_0}} |y'(x)|dx \leq c\varepsilon^{\frac{\alpha+1}{\beta+1}}+\varepsilon[a_{k_0}^{\alpha-\beta}+
a_{k(\varepsilon)}^{\alpha-\beta}]\leq c_2\varepsilon^{\frac{\alpha+1}{\beta+1}} ,
$$
for all $\varepsilon \in(0,\varepsilon_0)$. By \cite[Lemma 2.2.]{mersat} it follows that
$$
\M^{*d}(G(y)) <\infty  \quad \mbox{and} \quad \overline{\dim}_BG(y)   \leq d=2-\frac{\alpha+1}{\beta+1}.
$$
Finally, combining the obtained results, we have that the graph $G(y)$ is Minkowski nondegenerate, and
$$ \dim_BG(y)=2-\frac{\alpha+1}{\beta+1}=d.$$
\qed
\end{proof}

Now we can state a spiral-chirp comparison.

\begin{theorem}[Spiral-chirp comparison]{\label{c-s}}
Let $\alpha\in(0,1)$, and assume that  $x:[t_0,\infty)\to\eR$, $t_0>0$, is a function of class $C^2$, such that the planar curve $\Gamma=\{(x(t),\dot x(t)):
t\in[t_0,\infty)\}$
is a spiral $r=f(\f)$, $\f\in(\f_0,\infty)$, $\f_0>0$, in polar coordinates, near the origin, such that $f(\f)\simeq_1 \f^{-\alpha}$,  as $\f\to\infty$,
and $\dot\f(t)\simeq 1$,  as $t \to \infty$,
where $\f(t)$ is a function of class $C^1$ defined by
$\tan \f(t)=\frac{\dot x(t)}{x(t)}$.
Define $X(\tau)=x(1/\tau)$. Then $X=X(\tau)$ is  $(\alpha,1)$-chirp-like function, and
$$
\dim_{osc}(x):=\dim_BG(X)=\frac{3-\alpha}2,
$$
where $G(X)$ is graph of the function $X$. Furthermore, $G(X)$ is Minkowski nondegenerate.
\end{theorem}

\begin{proof}
Let us write the function $X(\tau)$ in the form
$$
X(\tau)=p(\tau)\cos q(\tau),\quad \tau\in(0,\frac{1}{t_0}],
$$
where
$$ p(\tau)=f(\f(\frac{1}{\tau})),  \quad q(\tau)=\f(\frac{1}{\tau}).$$
The function $p(\tau)$ is increasing near $\tau=0$ since $\frac{1}{\tau}$ is decreasing, $\f(t)$ is increasing and $f(\f)$ is decreasing near $\f=\infty$. Furthermore, $p\in C([0,1/t_0])$ since $p(0)=\lim_{\tau \to 0}f(\f(1/\tau))=0$,
by noting that $\dot \f \simeq  1$ implies $\f(t) \to \infty$ as $t\to \infty$. Now, we shall exploit Theorem \ref{BDchirp}, by checking that its assumptions are satisfied with $S(q)=\cos q$ and $\beta=1$.
The functions $\f$, $ p$ and $q$ have the following properties:
$$
\f(t)\simeq t \quad\textrm{as}\quad t\to\infty \quad \textrm{or}\quad \f(\frac{1}{\tau})\simeq \frac{1}{\tau} \quad \textrm{as}\quad \tau \to 0 ,
$$
\begin{eqnarray}
p(\tau) & \simeq_1 & \tau^\alpha \quad \textrm{as} \quad \tau\to 0, \nonumber\\
q(\tau) & \simeq_1 & \frac{1}{\tau}\quad \textrm{as}\quad \tau\to 0, \nonumber\\
q^{-1}(t) & \simeq & \frac{1}{t}\quad \textrm{as}\quad t\to\infty. \nonumber
\end{eqnarray}
The function $q$ is decreasing near the origin, thus $q^{-1}$ exists for $t$ large enough. We see that all conditions of Theorem \ref{BDchirp} are fulfilled.
\qed
\end{proof}

\begin{remark}
Theorem \ref{c-s} is a new version of \cite[Theorem 4]{chirp}.
If we compare Theorems \ref{s-c-p} and \ref{c-s} in terms of conditions, then we see that Theorem \ref{c-s} requires derivatives of lower order than Theorem \ref{s-c-p}. Phase plane already gives us the information about the first derivative.
\end{remark}

The following result shows that rectifiable spirals generate rectifiable chirp-like functions.

\begin{theorem}{\label{eqcs1}} {\bf (Rectifiability of a chirp generated by a rectifiable spiral)}
Let $\alpha >1$, and assume that  $x:[t_0,\infty)\to\eR$, $t_0>0$, is a function of class $C^2$ such that the planar curve $\Gamma=\{(x(t),\dot x(t)):
t\in[t_0,\infty)\}$
is a rectifiable spiral $r=f(\f)$, $\f\in(\f_0,\infty)$, $\f_0>0$ in polar coordinates, near the origin, such that $f(\f)\simeq_1 \f^{-\alpha}$, as $\f\to\infty$, $|f''(\f)|\le C\f^{-\alpha-2}$ and $\dot\f(t)\simeq1$ as $t\to\infty$, where $\f(t)$ is a function of class $C^1$ defined by
$\tan \f(t)=\frac{\dot x(t)}{x(t)}$.
Define $X(\tau)=x(1/\tau)$. Then $X=X(\tau)$ is $(\alpha,1)$-chirp-like rectifiable function near the origin.
\end{theorem}

To prove the theorem we shall use the following two lemmas.
\begin{lemma}\label{unique0}
Let $F,  G\in C^1(I)$, where $I$ is an open interval in $\eR$, and assume that $\inf F'>\sup G'$. Then the equation $F(z)= G(z)$
has at most one solution.
\end{lemma}

\begin{proof}
Suppose that there are two different solutions $z_1$ and $z_2$. Then applying
the mean-value theorem to $F(z_1)-F(z_2)=G(z_1)- G(z_2),$ we obtain that there exist $\tilde z_1$ and $\tilde z_2$ such that $F'(\tilde z_1)= G'(\tilde z_2)$. Therefore, $\inf F'\le\sup G'$. This contradicts the condition $\inf F'>\sup G'$.
\qed
\end{proof}

\begin{lemma}\label{unique}
Let $F\in C^1(0,\infty)$ be such that $F(z)\sim az$ as $z\to\infty$ for some $a<0$. Assume that $\inf F'>-\infty$.
Then there exists a nonnegative integer $k_0$ such that for each $k\ge k_0$ the equation $\cot z=F(z)$ possesses the unique solution in
$J_k=(k\pi,(k+1)\pi)$.
\end{lemma}

\begin{proof}
Since $F(z)$ is continuous and $F(z)\sim az$ as $z\to\infty$, and $\cot z$ restricted to $J_k$ is  continuous function onto $\eR$, it follows that the equation $\cot z=F(z)$ possesses at least one solution $z_k$ on each interval $J_k$. We have to show that the solution is unique on each $J_k$ for all $k$ large enough.

Since $m=\inf F'>-\infty$, there exists $s_0\in(\pi/2,\pi)$ sufficiently close to $\pi$ such that $\cot'(s_0)=-(\sin s_0)^{-2}<m$.
The condition $F(z)\sim az$ implies that, given any fixed $b\in(a,0)$, there exists $M=M(b)>0$ such that $F(z)< bz$
for all $z\ge M$. Let us fix any such $b$.

Let $k_0$ be a nonnegative integer such that $b(k_0\pi)<\cot s_0$. It suffices to take $k_0 > (b\pi)^{-1} \cot s_0$. Taking $k_0$ still larger, we can achieve that $k_0\pi\ge M$. Hence, for $z\ge k_0\pi$ we have $F(z)< bz$. In particular,
$$
F(z)< bz\le b(k_0\pi)<\cot s_0.
$$
Since for $z\ge k_0\pi$ we have $F(z)<\cot s_0$, while $\cot z\ge \cot s_0$ for each $z\in J_k\setminus I_k$,
where $I_k=(k\pi+s_0,(k+1)\pi)$, then  all solutions of equation
$F(z)=\cot z$ for $z\ge k_0\pi$ are contained in
$\cup_{k\ge k_0}I_k$.

Let us define $ G(z)=\cot z$, and consider the equation $F(z)= G(z)$ on $I_k$ for any $k\ge k_0$.
We have
$$
\sup_{I_k} G'=\cot'(k_0\pi+s_0)=-(\sin s_0)^{-2}<\inf_{(0,\infty)}F'\le \inf_{I_k}F'.
$$
The unique solvability of $F(z)= G(z)$ on $I_k$ follows from Lemma~\ref{unique0}.
The equation is uniquely solvable on $J_k$ as well, since there are no solutions in $J_k\setminus I_k$.
\qed
\end{proof}

\begin{remark}\label{remark-dokaz-sp-ch}
The condition $F(z)\sim az$ as $z\to\infty$ in Lemma~\ref{unique} can be weakened. It suffices to assume that $F(z)<bz$ for some $b<0$ and for all $z$ large enough.
\end{remark}

\begin{remark}
The condition $\inf F'>-\infty$ in Lemma~\ref{unique} cannot be dropped. To see this, we
construct a function $y=F(z)$ by means of a sequence of lines $y=b_nz$, where $a<b_n<0$
and $b_n\to a$ as $n\to\infty$. We first construct a continuous function $F_0$ such that on $J_k'=(k\pi,(k+1)\pi]$,
$$
F_0(z)=
\left\{
\begin{array}{ll}
b_kz,& \mbox{for $z\in(k\pi,z_k]$},\\
\cot z,& \mbox{for $z\in(z_k,v_k]$},\\
b_{k+1}z,& \mbox{for $z\in(v_k,(k+1)\pi$]},
\end{array}
\right.
$$
where $z_k$ and $v_k$ are respective solutions of equations $\cot z=b_kz$ and $\cot b_{k+1}v=b_{k+1}v$ in $J_k$.
The function $F_0$ is of class $C^1$ everywhere in $(0,\infty)$ except at the points $z_k$ and $v_k$.
We can perform its smoothing in sufficiently small neighborhoods of these points, in order to get a function $F\in C^1(0,\infty)$.
It is clear that $F(z)\sim az$ as $z\to\infty$ and $\inf F'=-\infty$. But $F(z)=\cot z$ possesses infinitely many solutions on each interval~$I_k$.
\end{remark}

\begin{remark}
Assume that
$$
F(z)=\frac{f(z)}{f'(z)},
$$
where $f\in C^2(0,\infty)$. (a) The condition $\inf F'>-\infty$
is equivalent to $f(z)f''(z)\le C f'{^2}(z)$, where $C$ is a positive constant. (b) The condition
$F(z)<bz$ for $z$ sufficiently large, where $b$ is a negative constant (see Remark \ref{remark-dokaz-sp-ch}), is satisfied
if for all $z$ large enough we have $f(z)\ge az^{-\alpha}$ and $f'(z)\ge a_1z^{-\alpha-1}$, where  $a>0$ and $a_1<0$ are constants. It suffices to take $b\in(a/a_1,0)$.
\end{remark}

A variation of Lemma \ref{unique} is the following lemma.
\begin{lemma}\label{unique1}
Let $F\in C^1(0,\infty)$ be such that $F(z)\sim az$ as $z\to\infty$ for some $a>0$. Assume that $\sup F'<\infty$.
Then there exists a nonnegative integer $k_0$ such that for each $k\ge k_0$ the equation $\tan z=F(z)$ possesses the unique solution in
$J_k=((k-1/2)\pi,(k+1/2)\pi)$.
\end{lemma}

\begin{remark}\label{remark-za-drpodt}
The condition $F(z)\sim az$ as $z\to \infty$ for $a>0$ in Lemma \ref{unique1} can be weakened by assuming that $F(z)>az$ for some $a>0$ and for all $z$ large enough. If $F(z)$
has the form $F(z)=\frac{f(z)}{f'(z)}$, where $f\in C^2(0,\infty)$, the condition $\sup F'<\infty$ is equivalent to $f(z)f''(z)\geq Cf'^2(z)$, where $C$ is positive constant.
Also, in that case, the condition $F(z)>az$ for $z$ large enough is satisfied if for all $z$ large enough we have $f(z)\geq a_1z^{-\alpha}$ and $f'(z)\leq a_2z^{-\alpha-1}$,
where $a_1$ and $a_2$ are positive constants. It suffices to take $a\in (0,\frac{a_1}{a_2})$.
\end{remark}

\begin{proof}[Theorem \ref{eqcs1}]
We can write the function $X(\tau)$ in the form $X(\tau)=p(\tau)\cos q(\tau)$, where $p(\tau)=f(\f(1/\tau))\simeq \tau^\alpha$, $p'(\tau)\simeq\tau^{\alpha-1}$,
$q(\tau)=\f(1/\tau)\simeq \tau^{-1}$, $q'(\tau)\simeq -\tau^{-2}$ as $\tau \to 0$. It follows that $X$ is an $(\alpha,1)$-chirp-like function.
Using the assumptions of the theorem, for the function
$$F(t):=\frac{pq'}{p'}(q^{-1}(t))=\frac{f(t)}{f'(t)}$$
we have $F(t)\simeq -t$ as $t\to\infty$, and $\frac{f(t)f''(t)}{f'^2(t)}<C$, for $t$ large enough, $C>0$. Then there exists $k_0\in \eN$ such that the equation $\cot q(t) =F(q(t))=\frac{p(\tau)q'(\tau)}{p'(\tau)}$
has the unique solution  $s_k\in(a_{k+1},a_k)$ where $a_{k+1}=q^{-1}((2k+1)\frac{\pi}{2}$ and $a_k=q^{-1}((2k-1)\frac{\pi}{2})$ for all $k\ge k_0$, see Lemma \ref{unique} and Remark \ref{remark-dokaz-sp-ch}. These solutions are just points of local extrema of $X(\tau)$ on $(a_{k+1},a_k)$, $k\ge k_0$. The sequence $(a_k)_{k\geq 1}$ of zero points of $X$ on $(0,1/t_0]$ is decreasing. Hence the sequence $(s_k)$ of
consecutive points of local extrema of $X$ is also decreasing. We have that $a_k=q^{-1}((2k-1)\frac{\pi}{2})\simeq k^{-1}$ as $k\to\infty$. So the same is true also for $s_k$, i.e., $s_k\simeq k^{-1}$ as $k\to\infty$, and
we also have $|X(s_k)|\le p(s_k)\le Cs_k^\alpha\le C_1k^{-\alpha}$. This implies that
\begin{equation}\label{red}
\sum_{k=k_0}^\infty|X(s_k)|\le C_1\sum_{k=k_0}^\infty k^{-\alpha}<\infty
\end{equation}
for $\alpha>1$. The length of the graph $G(X)$ is defined by
$$ \mathrm{length}(G(X)):= \sup \sum_{i=1}^m \|(t_i,X(t_i))-(t_{i-1},X(t_{i-1}))\|_2,$$
where the supremum is taken over all partitions $0=t_0<t_1<\ldots <t_m=1/t_0$ of the interval $[0,1/t_0]$ and where $\| . \|_2$ denotes the Euclidean norm in $\eR^2$.
Using \cite[Lemma 3.1.]{mervanraguz}, it follows that
$\mathrm{length}(G(X))\le 2 \sum_k|X(s_k)| +1/t_0 $.
Then $X$ is rectifiable due to (\ref{red}).
\qed
\end{proof}

\begin{remark}
Theorem \ref{c-s} can be applied to planar systems with pure imaginary pair of eigenvalues, because the normal forms of such systems satisfy the conditions of Theorem \ref{c-s},
see \cite[Theorem 9]{zuzu}. Also, Theorem \ref{c-s} can be applied to related second order differential equations.
\end{remark}

\appendix
\section{Proofs of some technical results}

\begin{proof}[Lemma \ref{tehnicka2}]
From the assumption that $p(t)\sim_1 t^{-\alpha}$ as $t\to\infty$ we have that for each $\varepsilon >0$
there exist $\overline{t}_0\geq t_0$ such that for all $t\geq \overline{t}_0$,
$$
(1-\varepsilon)t^{-\alpha}<p(t)<(1+\varepsilon)t^{-\alpha},\quad (1-\varepsilon)\alpha t^{-\alpha-1}<-p'(t)<(1+\varepsilon)\alpha t^{-\alpha-1},
$$
$$
\frac{p'(t)}{p(t)} >-\frac{K_1\alpha}{t},
$$
where $K_1:=\frac{1+\varepsilon}{1-\varepsilon}$. Then we have
$$
r_{\min}(t)\leq r(t) \leq r_{\max}(t),\quad \mbox{for all}\quad t\in[\overline{t}_0,\infty),
$$
where
\begin{eqnarray}
r_{\min}(t) & := & p(t)\sqrt{1-K_1\alpha t^{-1}}, \nonumber\\
r_{\max}(t) & := & p(t)\sqrt{1+K_1\alpha t^{-1}+K_1^2\alpha^2 t^{-2}} , \nonumber
\end{eqnarray}
and without loss of generality we assume that $\overline{t}_0>K_1\alpha$. Therefore
\begin{eqnarray}
& & r(t(\f))-r(t(\f+\triangle \f)) \geq r_{\min}(t_1)-r_{\max}(t_2)= \nonumber \\
& & \frac{p^2(t_1)(1-K_1\alpha t_1^{-1})-p^2(t_2)(1+K_1\alpha t_2^{-1}+K_1^2\alpha^2t_2^{-2})}{p(t_1)\sqrt{1-K_1\alpha t_1^{-1}}+p(t_2)\sqrt{1+K_1\alpha t_2^{-1}+K_1^2\alpha^2 t_2^{-2}}},\nonumber
\end{eqnarray}
where
$$
t_1=t(\f)=\f(1+O(\f^{-1})), \quad t_2=t(\f+\triangle \f))=\f(1+O(\f^{-1})) \quad\textrm{as}\quad t\to\infty .
$$
Now, using relations
\begin{eqnarray}
p(t_1)\sqrt{1-K_1\alpha t_1^{-1}} & \simeq & \f^{-\alpha},\quad p(t_2)\sqrt{1+K_1\alpha t_2^{-1}+K_1^2\alpha^2 t_2^{-2}}\simeq \f^{-\alpha}, \nonumber\\
t_1^{-\alpha} & \simeq & \f^{-\alpha},\quad t_2^{-\alpha}\simeq \f^{-\alpha} \quad \mathrm{as}\ \f \to\infty ,\nonumber
\end{eqnarray}
\begin{eqnarray}
p(t_1)\sqrt{1-K_1\alpha t_1^{-1}} & + & p(t_2)\sqrt{1+K_1\alpha t_2^{-1}+K_1^2\alpha^2 t_2^{-2}}\leq 2C_1\f^{-\alpha},\nonumber\\
p^2(t_1) & - & p^2(t_2)\geq 2(1-\varepsilon)^2\triangle \f \alpha\f^{-2\alpha-1}(1+O(\f^{-1})),\nonumber\\
-K_1 \alpha(p^2(t_1)t_1^{-1} & + & p^2(t_2)t_2^{-1})\geq -2K_1\alpha (1+\varepsilon)^2\f^{-2\alpha-1}(1+O(\f^{-1})), \nonumber\\
& - & K_1^2\alpha^2p^2(t_2)t_2^{-2}=O(\f^{-2\alpha-2}), \nonumber
\end{eqnarray}
where $C_1>0$ is independent of $\f$, we have
$$
r(t(\f))- r(t(\f+\triangle \f)) \geq \frac{\alpha(1-\varepsilon)^2}{C_1}\left[\triangle \f-\left(\frac{1+\varepsilon}{1-\varepsilon}\right)^3\right]\f^{-\alpha-1}(1+O(\f^{-1})) .
$$
If $\triangle \f>1$, then choosing $\varepsilon<\frac{(\triangle \f)^{1/3}-1}{(\triangle \f)^{1/3}+1}$, we obtain the claim.
\qed
\end{proof}

\begin{proof}[Proposition \ref{druga}]
From (\ref{apsp}) and (\ref{qq}) inequalities (\ref{qq1}) and (\ref{qq2}) follow directly by definition. The function $q|_{(0,\delta_0]}$ is
positive and decreasing, and its inverse function is defined on $[m_0,\infty)$. Relation  (\ref{qq-1}) follows from (\ref{qq2}) applying the well known formula for derivative of the inverse function. Then exploiting the mean value theorem and (\ref{qq-1}), we get (\ref{qq-1st}).
\qed
\end{proof}

\begin{proof}[Proposition \ref{treca}]
The claim in (i) is evident. To prove (ii), it suffices to take $k_0\in \eN$ such that $k_0T\geq m_0$. We shall prove inequality (\ref{pro4}) only, because
other facts are easy consequences of Proposition \ref{druga}. From (\ref{apsp}) we obtain that
$p(x)$ is positive and increasing function near $x=0$, and we have
$$
\max_{x\in[a_{k+1},a_k]} |y(x)| \geq p(s_k)|S(q(s_k))| \geq c p(a_{k+1}) \geq c_1 (a_{k+1})^\alpha \geq c_0 (k+1)^{-\frac{\alpha}{\beta}} ,
$$
for all $k\geq k_0$, where $c=\min \{ |S(t_0)|,|S(t_0+T)|\}$, $c_1=cC_1$ and $c_0=cC_1^2$ are positive constants. Now we prove (iii). Let $\varepsilon>0$ and let $k(\varepsilon)\in \eN$
be such that
$$
k(\varepsilon)\geq \frac{1}{T}\left(\frac{\varepsilon}{TC_2} \right)^{-\frac{\beta}{\beta+1}}=c\varepsilon^{-\frac{\beta}{\beta+1}},\quad c=T^{-1}(TC_2)^{\frac{\beta}{\beta+1}} .
$$
Let $\varepsilon_0'$ be such that for all $0<\varepsilon \leq \varepsilon_0' $ it holds $k(\varepsilon)T\geq m_0=q(\delta_0)$.
Further, for all $ \varepsilon<c^{\frac{\beta+1}{\beta}}$ we have  $2c\varepsilon^{-\frac{\beta}{\beta+1}}-c\varepsilon^{-\frac{\beta}{\beta+1}}>1$. So, there exists $k(\varepsilon)\in \eN$ such that
$$1<c\varepsilon^{-\frac{\beta}{\beta+1}}\leq k(\varepsilon)\leq 2c\varepsilon^{-\frac{\beta}{\beta+1}},\quad \mbox{for all}\quad \varepsilon<c^{\frac{\beta+1}{\beta}}.$$
Let us take $\varepsilon_0=\min \{\varepsilon_0',c^{\frac{\beta+1}{\beta}}\}$. Then we can find $k(\varepsilon)\in \eN$ such that
$$c\varepsilon^{-\frac{\beta}{\beta+1}}\leq k(\varepsilon)\leq 2c\varepsilon^{-\frac{\beta}{\beta+1}},\quad k(\varepsilon)T\geq m_0 \quad \mbox{for all}\quad
\varepsilon\in(0,\varepsilon_0).$$
Using (\ref{qq-1st}), then for all $k\geq k(\varepsilon)$ and $\varepsilon\in(0,\varepsilon_0)$ it holds
$$
C_1T((k+1)T)^{-\frac{1}{\beta}-1}\leq a_k-a_{k+1}\leq \varepsilon .
$$
\qed
\end{proof}

\bibliographystyle{plain}
\bibliography{bibliografija}

\def\cprime{$'$}
\begin{thebibliography}{10}

\bibitem{arnold-vol2}
V.~I. Arnol{\cprime}d, S.~M. Guse{\u\i}n-Zade, and A.~N. Varchenko.
\newblock {\em Singularities of differentiable maps. {V}ol. {II}}, volume~83 of
  {\em Monographs in Mathematics}.
\newblock Birkh\"auser Boston Inc., Boston, MA, 1988.
\newblock Monodromy and asymptotics of integrals, Translated from the Russian
  by Hugh Porteous, Translation revised by the authors and James Montaldi.

\bibitem{tric}
Yves Dupain, Michel Mend{\`e}s~France, and Claude Tricot.
\newblock Dimensions des spirales.
\newblock {\em Bull. Soc. Math. France}, 111(2):193--201, 1983.

\bibitem{falc}
Kenneth Falconer.
\newblock {\em Fractal geometry}.
\newblock John Wiley \& Sons Ltd., Chichester, 1990.
\newblock Mathematical foundations and applications.

\bibitem{laho}
Lana Horvat~Dmitrovi{\'c}.
\newblock Box dimension and bifurcations of one-dimensional discrete dynamical
  systems.
\newblock {\em Discrete Contin. Dyn. Syst.}, 32(4):1287--1307, 2012.

\bibitem{lukamaja}
Luka Korkut and Maja Resman.
\newblock Oscillations of chirp-like functions.
\newblock {\em Georgian Math. J.}, 19(4):705--720, 2012.

\bibitem{q_clothoid}
Luka Korkut, Domagoj Vlah, Darko {\v{Z}}ubrini{\'c}, and Vesna
  {\v{Z}}upanovi{\'c}.
\newblock Generalized {F}resnel integrals and fractal properties of related
  spirals.
\newblock {\em Appl. Math. Comput.}, 206(1):236--244, 2008.

\bibitem{bessel}
Luka Korkut, Domagoj Vlah, and Vesna {\v{Z}}upanovi{\'c}.
\newblock Fractal properties of {B}essel functions.
\newblock arXiv:1304.1762, Unpublished results.

\bibitem{luvedo}
Luka Korkut, Domagoj Vlah, and Vesna {\v{Z}}upanovi{\'c}.
\newblock Geometrical properties of a class of systems with spiral trajectories
  in {${\mathbb R}^3$}.
\newblock arXiv:1211.0918, Unpublished results.

\bibitem{kpw}
Man~Kam Kwong, Mervan Pa{\v{s}}i{\'c}, and James S.~W. Wong.
\newblock Rectifiable oscillations in second-order linear differential
  equations.
\newblock {\em J. Differential Equations}, 245(8):2333--2351, 2008.

\bibitem{MRZ}
Pavao Marde{\v{s}}i{\'c}, Maja Resman, and Vesna {\v{Z}}upanovi{\'c}.
\newblock Multiplicity of fixed points and growth of
  {$\varepsilon$}-neighborhoods of orbits.
\newblock {\em J. Differential Equations}, 253(8):2493--2514, 2012.

\bibitem{mzz-Schrodinger}
Josipa~Pina Mili{\v{s}}i{\'c}, Darko {\v{Z}}ubrini{\'c}, and Vesna
  {\v{Z}}upanovi{\'c}.
\newblock Fractal analysis of {H}opf bifurcation for a class of completely
  integrable nonlinear {S}chr\"odinger {C}auchy problems.
\newblock {\em Electron. J. Qual. Theory Differ. Equ.}, pages No. 60, 32, 2010.

\bibitem{pasiceuler}
Mervan Pa{\v{s}}i{\'c}.
\newblock Rectifiable and unrectifiable oscillations for a class of
  second-order linear differential equations of {E}uler type.
\newblock {\em J. Math. Anal. Appl.}, 335(1):724--738, 2007.

\bibitem{pasic}
Mervan Pa{\v{s}}i{\'c}.
\newblock Fractal oscillations for a class of second order linear differential
  equations of {E}uler type.
\newblock {\em J. Math. Anal. Appl.}, 341(1):211--223, 2008.

\bibitem{mervanraguz}
Mervan Pa{\v{s}}i{\'c} and Andrija Ragu{\v{z}}.
\newblock Rectifiable oscillations and singular behaviour of solutions of
  second-order linear differential equations.
\newblock {\em Int. J. Math. Anal. (Ruse)}, 2(9-12):477--490, 2008.

\bibitem{mersat}
Mervan Pa{\v{s}}i{\'c} and Satoshi Tanaka.
\newblock Fractal oscillations of self-adjoint and damped linear differential
  equations of second-order.
\newblock {\em Appl. Math. Comput.}, 218(5):2281--2293, 2011.

\bibitem{pasicwo}
Mervan Pa{\v{s}}i{\'c} and James S.~W. Wong.
\newblock Rectifiable oscillations in second-order half-linear differential
  equations.
\newblock {\em Ann. Mat. Pura Appl. (4)}, 188(3):517--541, 2009.

\bibitem{chirp}
Mervan Pa{\v{s}}i{\'c}, Darko {\v{Z}}ubrini{\'c}, and Vesna
  {\v{Z}}upanovi{\'c}.
\newblock Oscillatory and phase dimensions of solutions of some second-order
  differential equations.
\newblock {\em Bull. Sci. Math.}, 133(8):859--874, 2009.

\bibitem{RaZuZu}
Goran Radunovi{\'c}, Darko {\v{Z}}ubrini{\'c}, and Vesna {\v{Z}}upanovi{\'c}.
\newblock Fractal analysis of {H}opf bifurcation at infinity.
\newblock {\em International Journal of Bifurcation and Chaos},
  22(12):1230043--1--1230043--15, 2012.

\bibitem{majaformal}
Maja Resman.
\newblock Epsilon-neighborhoods of orbits and formal classification of
  parabolic diffeomorphisms.
\newblock {\em Discrete Contin. Dyn. Syst.}, 33(8):3767--3790, 2013.

\bibitem{tricot}
Claude Tricot.
\newblock {\em Curves and fractal dimension}.
\newblock Springer-Verlag, New York, 1995.
\newblock With a foreword by Michel Mend{\`e}s France, Translated from the 1993
  French original.

\bibitem{wong1}
James S.~W. Wong.
\newblock On rectifiable oscillation of {E}uler type second order linear
  differential equations.
\newblock {\em Electron. J. Qual. Theory Differ. Equ.}, pages No. 20, 12 pp.
  (electronic), 2007.

\bibitem{li}
Hao Wu and Weigu Li.
\newblock Isochronous properties in fractal analysis of some planar vector
  fields.
\newblock {\em Bull. Sci. Math.}, 134(8):857--873, 2010.

\bibitem{zuzu}
Darko {\v{Z}}ubrini{\'c} and Vesna {\v{Z}}upanovi{\'c}.
\newblock Fractal analysis of spiral trajectories of some planar vector fields.
\newblock {\em Bull. Sci. Math.}, 129(6):457--485, 2005.

\bibitem{cras}
Darko {\v{Z}}ubrini{\'c} and Vesna {\v{Z}}upanovi{\'c}.
\newblock Fractal analysis of spiral trajectories of some vector fields in
  {${\mathbb R}^3$}.
\newblock {\em C. R. Math. Acad. Sci. Paris}, 342(12):959--963, 2006.

\bibitem{zuzulien}
Darko {\v{Z}}ubrini{\'c} and Vesna {\v{Z}}upanovi{\'c}.
\newblock Poincar\'e map in fractal analysis of spiral trajectories of planar
  vector fields.
\newblock {\em Bull. Belg. Math. Soc. Simon Stevin}, 15(5, Dynamics in
  perturbations):947--960, 2008.

\bibitem{fdd}
Vesna {\v{Z}}upanovi{\'c} and Darko {\v{Z}}ubrini{\'c}.
\newblock Fractal dimensions in dynamics.
\newblock In J.-P.\ Fran\ced{c}oise, G.L.\ Naber, and S.T.\ Tsou, editors, {\em
  Encyclopedia of Mathematical Physics}, volume~2, pages 394--402. Elsevier,
  Oxford, 2006.

\end{thebibliography}

\end{document}